\documentclass[a4paper]{article}

\usepackage[a4paper,margin=2.5cm]{geometry}
\usepackage{amsmath,amssymb,amsfonts,amsthm}
\usepackage{mathrsfs}
\usepackage{graphicx}
\usepackage[colorlinks=true]{hyperref}
\usepackage{indentfirst}

\numberwithin{equation}{section}

\newtheorem{Thm}{Theorem}[section]

\newtheorem{Lem}{Lemma}[section]
\newtheorem{Pro}{Proposition}[section]

\newtheorem{Rem}{Remark}[section]

\newcommand{\N}{\mathbb{N}}

\newcommand{\R}{\mathbb{R}}

\title{A positive ground state for a planar Choquard equation with mixed diffusion and critical exponential growth}

\author{Shaoxiong Chen, Sekhar Ghosh, Vishvesh Kumar, Zhipeng Yang\thanks{Corresponding author: Z. Yang.}}

\date{}

\AtEndDocument{%
  \par
  \bigskip
  \bigskip

  \noindent
  \textbf{Shaoxiong Chen:}\\[0.2em]
  \textsc{Department of Mathematics, Yunnan Normal University, Kunming, China}\\[0.3em]
  \textit{E-mail address}: \texttt{gxmail@126.com}\\[1.5em]
  \noindent
  \textbf{Sekhar Ghosh:}\\[0.2em]
  \textsc{Department of Mathematics, National Institute of Technology Calicut, Kerala, India}\\[0.3em]
  \textit{E-mail address}: \texttt{sekharghosh@nitc.ac.in and sekharghosh1234@gmail.com}\\[1.5em]
  \noindent
  \textbf{Vishvesh Kumar:}\\[0.2em]
  \textsc{Department of Mathematical Sciences, Indian Institute of Technology (BHU), Varanasi, India.}\\[0.3em]
  \textit{E-mail address}: \texttt{vishvesh.mat@iitbhu.ac.in and vishveshmishra@gmail.com}\\[1.5em]
  \noindent
  \textbf{Zhipeng Yang:}\\[0.2em]
  \textsc{Department of Mathematics, Yunnan Normal University, Kunming, China}\\
  \textsc{Yunnan Key Laboratory of Modern Analytical Mathematics and Applications, Kunming, China}\\[0.3em]
  \textit{E-mail address}: \texttt{yangzhipeng326@163.com}%
}

\begin{document}

\maketitle

\begin{abstract}
We study a two-dimensional Choquard equation driven by the mixed local and nonlocal operator \(L:=-\Delta+(-\Delta)^s\), where the nonlinearity has critical exponential growth of Trudinger--Moser type. Under a coercive assumption on the potential and suitable
one-sided assumptions on the nonlinearity, we prove the existence of a least energy positive solution. The proof combines Nehari manifold minimization, compactness below the critical Trudinger--Moser threshold, local regularity, and a strong maximum principle.

\end{abstract}

\section{Introduction and main result}

In this paper, we study the existence of a least energy positive solution for the planar Choquard equation with mixed local and nonlocal diffusion
\begin{equation}\label{eq1.1}
-\Delta u+(-\Delta)^su+V(x)u
=\left(\frac1{|x|^\mu}*F(u)\right)f(u)
\quad\text{in }\R^2,
\end{equation}
where \(s\in(0,1)\), \(\mu\in(0,2)\), \(F(t)=\int_0^t f(\tau)\,d\tau,\) and \(*\) denotes the convolution in \(\R^2\). The mixed operator
\[
\mathcal L:=-\Delta+(-\Delta)^s
\]
combines the classical Laplacian with a fractional Laplacian. For smooth functions, the fractional Laplacian is given, up to a normalization constant, by
\[
(-\Delta)^su(x):=\operatorname{P.V.}\int_{\R^2}\frac{u(x)-u(y)}{|x-y|^{2+2s}}\,dy.
\]
Thus \(\mathcal L\) describes the superposition of a local Brownian diffusion and a long-range jump process. Mixed operators of this form have recently attracted attention because they display features that are absent in purely local or purely nonlocal models. In particular, Biagi, Dipierro, Valdinoci and Vecchi  \cite{Biagi2022CPDE} developed a systematic analysis of
\[
-\Delta u+(-\Delta)^su=f
\]
including existence, maximum principles, and regularity for mixed local--nonlocal elliptic problems. Related recent developments include a Brezis--Oswald approach for mixed operators \cite{Biagi2024CCM}, combined local--nonlocal stationary problems \cite{Arora2025PRSE}, and Lazer--McKenna type problems for mixed elliptic operators \cite{HuangHajaiej2025LazerMcKennaMixed}. These results are essential for problems such as \eqref{eq1.1}, where local elliptic regularization and nonlocal interactions occur simultaneously.

We note that the right-hand side of \eqref{eq1.1} is a Choquard-type nonlinearity. In the classical setting, the model involving the Laplacian goes back to Pekar's description of polarons \cite{Pekar1954}, given by
\begin{equation} \label{eq1.2}
    -\Delta u+u=\left(\frac1{|x|}*u^2\right)u
\quad\text{in }\R^3.
\end{equation}
After that, Lieb \cite{Lieb1977} proved the existence and uniqueness, up to translations, of the minimizing solution of \eqref{eq1.2} by rearrangement methods. Lions \cite{Lions1980NA} later studied related Hartree--type equations and variational structures. For a modern treatment of nonlinear Choquard equations, Moroz and Van Schaftingen  \cite{MorozVanSchaftingen2013} investigated existence, regularity, positivity, symmetry, and decay properties of ground states of the following problem
\[
-\Delta u+u=(I_\alpha*|u|^p)|u|^{p-2}u\quad\text{in }\R^N,
\]
 and identified optimal ranges of \(p\). We cite the guide \cite{MorozVanSchaftingen2017} for a comprehensive survey of Choquard-type equations. We also refer to the monograph of Lieb and Loss for the Hardy--Littlewood--Sobolev inequality used to handle the convolution term \cite{LiebLoss2001}.

In dimensions \(N\ge3\), the critical Choquard exponent is governed by the Hardy--Littlewood--Sobolev inequality. However, in two dimensions, the Sobolev critical exponent is replaced by the Trudinger--Moser threshold. The sharp exponential growth is encoded by inequalities of the form
\[
\int_{\R^2}\bigl(e^{\alpha u^2}-1\bigr)\,dx<\infty
\quad\text{for }\alpha\le4\pi
\]
under a suitable normalization of \(u\). Cao's whole-space version of the Trudinger--Moser inequality \cite{Cao1992CPDE} and the Adachi--Tanaka inequality \cite{AdachiTanaka2000} are standard tools in variational problems with critical exponential growth. Related elliptic equations in \(\R^2\) with critical exponential nonlinearities and exact-growth refinements were studied in \cite{deFigueiredoMiyagakiRuf1995, IbrahimMasmoudiNakanishi2015}.

In the present paper, the convolution term forces the natural exponent \(p_\mu=\frac4{4-\mu},\) which is dictated by the Hardy--Littlewood--Sobolev relation in \(\R^2\). Consequently, compactness has to be recovered below a threshold depending on both the Trudinger--Moser constant \(4\pi\) and \(p_\mu\).

Recently, several works have focused on Choquard equations with mixed diffusion or critical behavior. For instance, Anthal, Giacomoni, and Sreenadh \cite{Anthal2023JMAA} studied mixed local--nonlocal Choquard equations on bounded domains with Hardy--Littlewood--Sobolev critical nonlinearities. Following this work, Giacomoni, Nidhi and Sreenadh \cite{GiacomoniNidhiSreenadh2025} considered normalized solutions for a critical Choquard equation involving mixed diffusion-type operators and proved existence and regularity results under mass constraints. In a semiclassical two-dimensional setting, Chen, Yang, and Yang \cite{ChenYangYang2026} studied the following problem
\[
-\varepsilon^2\Delta u+\varepsilon^{2s}(-\Delta)^su+V(x)u
=\varepsilon^{\mu-2}\left(\frac1{|x|^\mu}*F(u)\right)f(u)
\quad\text{in }\R^2,
\]
where \(f\) has Trudinger--Moser critical exponential growth, and obtained ground states together with concentration phenomena as \(\varepsilon\to0^+\). Furthermore, Chen, Hajaiej, Yang, and Yang \cite{ChenHajaiejYangYang2026} studied the fixed-scale mixed Choquard equation with Trudinger--Moser subcritical exponential growth and proved the existence of a least energy positive solution, a sign-changing solution, and infinitely many sign-changing solutions under an additional oddness assumption.

Inspired by these works, the present paper addresses the critical exponential regime for the fixed-scale problem \eqref{eq1.1} on the whole plane and establishes the existence of a least energy positive solution by a Nehari minimization scheme below the critical compactness threshold.

The main difficulty and novelty lies in handling the simultaneous presence of three noncompact effects: the unbounded domain \(\R^2\), the critical exponential growth, and the nonlocal Choquard convolution. The coercive condition on \(V\) restores compact embeddings of the corresponding Sobolev space into finite Lebesgue spaces, but this alone is not sufficient at the Trudinger--Moser threshold. We therefore combine a Moser-type comparison argument with a low-energy compactness lemma. More precisely, the critical lower bound in \((f_4)\) gives a path below the level \(\ell_* = \frac{a_\theta}{p_\mu},~ a_\theta=\frac12-\frac1{2\theta},\) and in particular, the Palais--Smale compactness is then proved below \(\ell_*\). This allows us to minimize the energy on the Nehari manifold, obtain a nontrivial nonnegative weak solution, and finally use local regularity and the strong maximum principle for mixed operators to conclude strict positivity.

We impose the following assumptions on the potential \(V\):
\begin{itemize}
\item[$(V_1)$] \(V\in C(\R^2,\R)\) and there exists \(V_0>0\) such that
\[
V(x)\ge V_0 \quad \text{for all } x\in\R^2.
\]
\item[$(V_2)$] For every \(M>0\),
\[
\mathrm{meas}\bigl(\{x\in\R^2: V(x)\le M\}\bigr)<\infty.
\]
\end{itemize}
For the Ambrosetti--Rabinowitz constant \(\theta>1\) appearing in \((f_3)\), set
\[
a_\theta=\frac12-\frac1{2\theta},
\qquad
\ell_*=\frac{a_\theta}{p_\mu},
\]
where \(p_\mu=\frac4{4-\mu}.\) We assume that the nonlinearity is truncated on the negative half-line and satisfies critical growth on the positive half-line:
\begin{itemize}
\item[$(f_1)$] \(f\in C^1(\R,\R)\), \(f(t)=0\) for all \(t\le0\), and \(f\) has critical exponential growth at \(+\infty\), namely
\[
\lim_{t\to+\infty}\frac{f(t)}{\exp(\alpha t^{2})}=0\quad\text{for all }\alpha>4\pi,
\qquad
\lim_{t\to+\infty}\frac{f(t)}{\exp(\alpha t^{2})}=+\infty\quad\text{for all }\alpha<4\pi.
\]
\item[$(f_2)$] The map \(t\mapsto \frac{f(t)}{t}\) is strictly increasing on \((0,+\infty)\).
\item[$(f_3)$] There exists \(\theta>1\) such that \(0\le \theta F(t)\le f(t)t\) for all \(t\ge0.\)
\item[$(f_4)$] There exist \(T_0>0\) and \(C_0\ge C_*\), where
\(C_*>0\) is the constant defined in Lemma~\ref{Lem3.5}, such that
\[
F(t)\ge C_0 e^{4\pi t^2}
\qquad\text{for all }t\ge T_0.
\]
Here \(C_*\) may depend on \(T_0\) and on the auxiliary Moser construction used in Lemma~\ref{Lem3.5}.

\end{itemize}

\begin{Rem}
The smallness of $f$ at the origin needed below follows from \((f_1)\). Indeed, since
\(f\in C^1(\R,\R)\) and \(f(t)=0\) for all \(t\le0\), we have
\(f(0)=f'(0)=0\). Hence
\[
f(t)=o(t)\quad\text{as }t\to0^+,
\]
and, because \((2-\mu)/2\in(0,1)\), we also have
\[
\frac{f(t)}{t^{(2-\mu)/2}}\to0
\quad\text{as }t\to0^+.
\]
\end{Rem}
We now state the main result of this paper.
\begin{Thm}\label{Thm1.1}
Assume that \(V\) satisfies \((V_1)\)--\((V_2)\) and \(f\) satisfies \((f_1)\)--\((f_4)\).
Then problem \eqref{eq1.1} admits a least energy positive solution.
\end{Thm}

The paper is organized as follows. In Section~\ref{Sec2} we set up the variational
framework and prove the compactness result below the critical threshold. In
Section~\ref{Sec3} we prove Theorem~\ref{Thm1.1}.

\section{The variational framework and critical compactness}\label{Sec2}

\noindent Throughout this section we assume that $V$ satisfies $(V_1)$--$(V_2)$ and that $f$ satisfies $(f_1)$--$(f_4)$. For any $p\in[1,+\infty)$, we write
\[
\|u\|_{p}:=\Bigl(\int_{\R^{2}}|u(x)|^{p}\,dx\Bigr)^{\frac{1}{p}}.
\]
The Sobolev space $H^{1}(\R^{2})$ is defined by
\[
H^{1}(\R^{2})
:=\bigl\{u\in L^{2}(\R^{2}) : \nabla u\in L^{2}(\R^{2};\R^{2})\bigr\}
\]
endowed with the norm
\[
\|u\|_{H^{1}(\R^{2})}
:=\Bigl(\int_{\R^{2}}\bigl(|u|^{2}+|\nabla u|^{2}\bigr)\,dx\Bigr)^{\frac12}.
\]
For $s\in(0,1)$, the fractional Sobolev space $H^{s}(\R^{2})$ is defined by
\[
H^{s}(\R^{2})
=\Bigl\{u\in L^{2}(\R^{2}) :
\int_{\R^{2}}\int_{\R^{2}}\frac{|u(x)-u(y)|^{2}}{|x-y|^{2+2s}}\,dx\,dy<\infty\Bigr\},
\]
with the Gagliardo seminorm
\[
[u]_{s}
=\Bigl(\int_{\R^{2}}\int_{\R^{2}}\frac{|u(x)-u(y)|^{2}}{|x-y|^{2+2s}}\,dx\,dy\Bigr)^{\frac12}.
\]
We write $\langle \cdot,\cdot\rangle$ for the duality pair between $H^{-s}(\R^2)$ and $H^{s}(\R^2)$, and fix the normalization of $(-\Delta)^s$ by
\[
\langle (-\Delta)^{s}u,v\rangle
=\frac12\int_{\R^{2}}\int_{\R^{2}}
\frac{(u(x)-u(y))(v(x)-v(y))}{|x-y|^{2+2s}}\,dx\,dy~~\text{for all}~~u,v\in H^{s}(\R^{2}).
\]
The following lemma can be found in \cite{Anthal2023JMAA}.
\begin{Lem}\label{Lem2.1}
Let $0<s<1$. Then $H^{1}(\R^{2})$ is continuously embedded into $H^{s}(\R^{2})$, that is, there exists a constant $C_{s}>0$ such that for every $u\in H^{1}(\R^{2})$,
\[
[u]_{s}^{2}\le C_{s}\,\|u\|_{H^{1}(\R^{2})}^{2}.
\]
\end{Lem}
We introduce the space
\[
E=\left\{u \in H^1\left(\R^2\right): \int_{\R^2} V(x)|u|^2\,dx<+\infty\right\},
\]
equipped with the inner product
\[
(u,v)
=\int_{\R^{2}}\nabla u\cdot\nabla v\,dx
+\frac{1}{2}\int_{\R^{2}}\int_{\R^{2}}
\frac{(u(x)-u(y))(v(x)-v(y))}{|x-y|^{2+2s}}\,dx\,dy
+\int_{\R^{2}}V(x)\,u(x)v(x)\,dx,
\]
and norm
\[
\|u\|^{2}
=\int_{\R^{2}}|\nabla u|^{2}\,dx
+\frac{1}{2}\int_{\R^{2}}\int_{\R^{2}}
\frac{|u(x)-u(y)|^{2}}{|x-y|^{2+2s}}\,dx\,dy
+\int_{\R^{2}}V(x)\,|u(x)|^{2}\,dx.
\]
By Lemma~\ref{Lem2.1} and $(V_1)$, the above norm is well defined on $E$, and $E$ is a Hilbert space.

\begin{Rem}
In view of $(V_1)$--$(V_2)$, the space $E$ is compactly embedded in $L^{p}(\R^2)$ for any $2 \le p<+\infty$. This compactness property is standard for coercive potentials; we refer to \cite{kondrat1999discreteness} for more details.
\end{Rem}

\begin{Lem}\label{Lem2.2}
Assume \((V_1)\). Then \(C_c^\infty(\R^2)\) is dense in \(E\).
\end{Lem}

\begin{proof}
Let \(u\in E\). Choose \(\chi\in C_c^\infty(\R^2)\) such that
\(0\le \chi\le1\), \(\chi\equiv1\) on \(B_1\), and \(\chi\equiv0\) on
\(\R^2\setminus B_2\). For \(R>1\), we set \(\chi_R(x)=\chi(x/R)\), so that \(\chi_R\in C_c^\infty(\R^2)\), \(\chi_R\equiv1\) on \(B_R\), \(\operatorname{supp}\chi_R\subset B_{2R}\), and
\(|\nabla\chi_R|\le C/R\) for some constant $C>0$.

We first show that \(\chi_Ru\to u\) in \(E\). Since \(u\in H^1(\R^2)\), one has
\[
\|(1-\chi_R)u\|_{L^2(\R^2)}\to0
\]
by dominated convergence, and
\[
\nabla\bigl((1-\chi_R)u\bigr)
=(1-\chi_R)\nabla u-u\nabla\chi_R.
\]
The first term converges to zero in \(L^2(\R^2)\) by dominated convergence, while
\[
\int_{\R^2}|u|^2|\nabla\chi_R|^2\,dx
\le \frac{C}{R^2}\int_{B_{2R}\setminus B_R}|u|^2\,dx
\le \frac{C}{R^2}\|u\|_{L^2(\R^2)}^2\to0.
\]
Thus, we have \(\chi_Ru\to u\) in \(H^1(\R^2)\). By Lemma~\ref{Lem2.1}, we get
\[
[\chi_Ru-u]_s\to0.
\]
Moreover, since \(Vu^2\in L^1(\R^2)\) and \(|1-\chi_R|\le1\), we obtain
\[
\int_{\R^2}V(x)\,|(1-\chi_R)u|^2\,dx\to0
\]
by dominated convergence. Therefore \(\chi_Ru\to u\) in \(E\).

We now fix \(R>1\). Since \(\chi_Ru\in H_0^1(B_{2R})\), there exists a sequence
\((\varphi_{R,n})\subset C_c^\infty(B_{2R})\) such that
\[
\varphi_{R,n}\to \chi_Ru
\qquad\text{in }H^1(\R^2).
\]
Again Lemma~\ref{Lem2.1} yields $[\varphi_{R,n}-\chi_Ru]_s\to0.$ Because \(V\) is continuous, it is bounded on \(B_{2R}\); and hence
\[
\int_{\R^2}V(x)\,|\varphi_{R,n}-\chi_Ru|^2\,dx
=
\int_{B_{2R}}V(x)\,|\varphi_{R,n}-\chi_Ru|^2\,dx
\le
\|V\|_{L^\infty(B_{2R})}\|\varphi_{R,n}-\chi_Ru\|_{L^2(\R^2)}^2\to0.
\]
Thus \(\varphi_{R,n}\to \chi_Ru\) in \(E\). Therefore, one can apply a diagonal argument to get a sequence
in \(C_c^\infty(\R^2)\) converging to \(u\) in \(E\).
\end{proof}
\noindent We say that $u\in E$ is a weak solution of \eqref{eq1.1} if 
\[
\int_{\R^{2}}\nabla u\cdot\nabla v\,dx
+\langle (-\Delta)^s u,v\rangle
+\int_{\R^{2}}V(x)\,u\,v\,dx
=\int_{\R^{2}}\left(\frac{1}{|x|^{\mu}}*F(u)\right)f(u)\,v\,dx,~~\text{for every}~~v\in E.
\]
The energy functional associated with \eqref{eq1.1} is
\[
I(u)
=\frac{1}{2}\|u\|^{2}
-\frac{1}{2}\int_{\R^{2}}\left(\frac{1}{|x|^{\mu}}*F(u)\right)F(u)\,dx.
\]
Note that if the nonlinear term involving convolution in $I$ is well-defined, its derivative is given by
\[
I'(u)[v]
=(u,v)
-\int_{\R^{2}}\left(\frac{1}{|x|^{\mu}}*F(u)\right) f(u)\,v\,dx.
\]
We record the following classical Hardy--Littlewood--Sobolev inequality (see \cite{LiebLoss2001}) that will be used repeatedly throughout the paper.

\begin{Lem}\label{Lem2.3}
Let $t,r>1$ and $0<\mu<N$ be such that
\[
\frac{1}{t}+\frac{\mu}{N}+\frac{1}{r}=2.
\]
If $g\in L^{t}(\R^{N})$ and $h\in L^{r}(\R^{N})$, then there exists $C(t,N,\mu,r)>0$ such that
\[
\left|
\int_{\R^{N}}\left(\frac{1}{|x|^{\mu}}*g\right)h\,dx
\right|
\le C(t,N,\mu,r)\,\|g\|_{t}\,\|h\|_{r}.
\]
In particular, when $N=2$ and $t=r=\frac{4}{4-\mu}$, one has
\[
\int_{\R^{2}}\left(\frac{1}{|x|^{\mu}}*G\right)G\,dx
\le C_\mu\,\|G\|_{\frac{4}{4-\mu}}^{2}.
\]
\end{Lem}
We also recall the following Trudinger--Moser inequality from \cite{Cao1992CPDE}.
\begin{Pro}\label{Pro2.1}
If $\alpha>0$ and $u\in H^{1}(\R^{2})$, then
\[
\int_{\R^{2}}\bigl(e^{\alpha u^{2}}-1\bigr)\,dx<+\infty.
\]
Moreover, if $\alpha<4\pi$ and $\|u\|_{2}\le M<+\infty$, then there exists a constant $C_{1}=C_{1}(M,\alpha)>0$ such that
\[
\sup_{\|\nabla u\|_{2}\le1,\ \|u\|_{2}\le M}
\int_{\R^{2}}\bigl(e^{\alpha u^{2}}-1\bigr)\,dx
\le C_{1}.
\]
\end{Pro}

\begin{Lem}\label{Lem2.4}
Assume \((V_1)\)--\((V_2)\) and \((f_1)\)--\((f_4)\). Then the functional \(I\)
is well defined on \(E\) and \(I\in C^1(E,\R)\). Moreover,
\[
I'(u)[v]
=(u,v)
-\int_{\R^{2}}\left(\frac{1}{|x|^{\mu}}*F(u)\right) f(u)\,v\,dx
\]
for all \(u,v\in E\).
\end{Lem}

\begin{proof}
We give the details only for the Choquard term. Let \(p_\mu=4/(4-\mu)\).
The growth assumptions imply that, for every \(\alpha>4\pi\), every \(q>2\),
and every \(\varepsilon>0\), there exists \(C>0\) such that
\begin{equation}\label{eq2.1}
|F(t)|\le \varepsilon |t|^{\frac{4-\mu}{2}}
+C|t|^q(e^{\alpha t^2}-1),
\end{equation}
and
\begin{equation}\label{eq2.2}
|f(t)|\le \varepsilon |t|^{\frac{2-\mu}{2}}
+C|t|^{q-1}(e^{\alpha t^2}-1)
\end{equation}
for all \(t\in\R\). Note that near the origin, these estimates follow from \((f_1)\). Indeed, since \(f\in C^1(\R,\R)\) and \(f(t)=0\) for \(t\le0\), we have \(f(0)=f'(0)=0\). Hence \(f(t)=o(t)\) and \(F(t)=o(t^2)\) as \(t\to0^+\). Since \((2-\mu)/2<1\) and \((4-\mu)/2<2\), this gives the required smallness at the origin. The estimates at infinity follow from the critical upper growth in \((f_1)\).

We first prove the following consequence of strong convergence in \(E\). If
\(u_n\to u\) in \(E\), then for every \(\gamma>0\) there exists \(r>1\) such that
\begin{equation}\label{eq2.3}
\sup_n\int_{\R^2}\bigl(e^{\gamma r u_n^2}-1\bigr)\,dx<+\infty .
\end{equation}
Indeed, fix \(\eta>0\) and let \(r>1\) and set \(z_n=u_n-u\). Using H\"older's inequality with exponents \(\rho,\rho'>1\) along with the following inequalities
\[
u_n^2\le (1+\eta)u^2+(1+\eta^{-1})z_n^2
\]
and
\[
e^{A+B}-1=(e^A-1)(e^B-1)+(e^A-1)+(e^B-1),\qquad A,B\ge0,
\]
we get
\[
\begin{aligned}
\int_{\R^2}(e^{\gamma r u_n^2}-1)\,dx
&\le C\|e^{\gamma r(1+\eta)u^2}-1\|_{\rho}
      \|e^{\gamma r(1+\eta^{-1})z_n^2}-1\|_{\rho'} \\
&\quad +C\int_{\R^2}(e^{\gamma r(1+\eta)u^2}-1)\,dx
      +C\int_{\R^2}(e^{\gamma r(1+\eta^{-1})z_n^2}-1)\,dx .
\end{aligned}
\]
Since \((e^t-1)^m\le C_m(e^{mt}-1)\) for \(t\ge0\) and
\(m>1\), by Proposition~\ref{Pro2.1} we conclude that
\[
e^{\gamma r(1+\eta)u^2}-1\in L^\rho(\R^2)
\]
for the fixed function \(u\). Since \(z_n\to0\) in \(H^1(\R^2)\), then for all
large \(n\), the functions
\[
w_n=\bigl(\gamma r\rho'(1+\eta^{-1})\bigr)^{1/2}z_n
\]
satisfy \(\|\nabla w_n\|_2\le1\) and have uniformly bounded \(L^2\)-norms.
Thus, the Trudinger--Moser inequality yields
\[
\sup_{n\ge n_0}\int_{\R^2}
\bigl(e^{\gamma r\rho'(1+\eta^{-1})z_n^2}-1\bigr)\,dx<+\infty .
\]
This controls the terms involving \(z_n\) uniformly for \(n\ge n_0\), while the
finitely many indices \(n<n_0\) are controlled by Proposition~\ref{Pro2.1}.
Therefore, the estimate \eqref{eq2.3} holds.

Let us now assume \(u\in E\). Applying \eqref{eq2.1}, the identity
\(p_\mu(4-\mu)/2=2\), H\"older's inequality and the Trudinger--Moser inequality,
we obtain
\begin{equation}\label{eq2.4}
    F(u)\in L^{p_\mu}(\R^2).
\end{equation}
Moreover, if \(v\in E\), then \(f(u)v\in L^{p_\mu}(\R^2)\). Indeed, set \(a=\frac4{2-\mu},
\quad
b=2.\)
Since
\[
\frac1{p_\mu}=\frac1a+\frac1b,
\qquad
a\frac{2-\mu}{2}=2,
\]
the growth estimate \eqref{eq2.2}, the Trudinger--Moser inequality,
and the embedding \(E\hookrightarrow L^p(\R^2)\) for every \(p\ge2\) imply that $f(u)\in L^a(\R^2).$
Since \(E\hookrightarrow L^2(\R^2)\), we obtain
\[
\|f(u)v\|_{p_\mu}
\le
\|f(u)\|_a\|v\|_2
\le C\|f(u)\|_a\|v\|<+\infty .
\]
Now, by Lemma~\ref{Lem2.3}, the functional \(I\) is well-defined, and is G\^ateaux differentiable. It remains to prove the continuity of the derivative. Let \(u_n\to u\) in \(E\).
By \eqref{eq2.1}, \eqref{eq2.3}, Vitali's theorem and the
strong convergence \(u_n\to u\) in every \(L^p(\R^2)\), \(2\le p<+\infty\), we get
\begin{equation}\label{eq2.5}
F(u_n)\to F(u)\qquad\text{in }L^{p_\mu}(\R^2).
\end{equation}
Similarly, with \(a=\frac4{2-\mu},\) we have
\[
f(u_n)\to f(u)\qquad\text{in }L^a(\R^2).
\]
Since \(1/p_\mu=1/a+1/2\), H\"older's inequality and the continuous embedding
\(E\hookrightarrow L^2(\R^2)\) yield
\[
\begin{aligned}
\sup_{\|v\|\le1}\|(f(u_n)-f(u))v\|_{p_\mu}
&\le
\|f(u_n)-f(u)\|_a\sup_{\|v\|\le1}\|v\|_2  \\
&\le C\|f(u_n)-f(u)\|_a\to0.
\end{aligned}
\]
Using \eqref{eq2.5} and Lemma~\ref{Lem2.3}, we conclude that
\(I'(u_n)\to I'(u)\) in \(E'\). Thus \(I\in C^1(E,\R)\).
\end{proof}

We now define the Nehari manifold as
\[
\mathcal{N}
:=\bigl\{u\in E\setminus\{0\} : I'(u)[u]=0\bigr\}.
\]
The following result is a standard consequence of the Lions-type concentration--compactness for the Trudinger--Moser inequality.
\begin{Lem}\label{Lem2.5}
Let $(u_n)\subset H^{1}(\R^{2})$ with $u_n\rightharpoonup u$ in $H^{1}(\R^{2})$.
Assume \(\limsup_{n\to\infty}\|\nabla u_n\|_{2}^{2}<1.\) Then there exist $q>1$ and $C>0$ such that
\[
\sup_{n\in\N}\int_{\R^{2}}\bigl(e^{4\pi q\,u_n^{2}}-1\bigr)\,dx\le C.
\]
\end{Lem}

\begin{proof}
Fix \(\rho=\limsup_{n\to\infty}\|\nabla u_n\|_2^2<1\) and choose
 $\beta>1$ such that $\beta\rho<1$.
Then there exists $n_0\in\N$ such that for all $n\ge n_0$,
\[
\|\nabla(\sqrt{\beta}\,u_n)\|_2^2=\beta\|\nabla u_n\|_2^2\le 1.
\]
Since $(u_n)$ is bounded in $H^{1}(\R^{2})$, there exists $M>0$ such that for all $n$,
\[
\|\sqrt{\beta}\,u_n\|_2\le M.
\]
Fix $\alpha\in(0,4\pi)$.
Applying Proposition~\ref{Pro2.1} to $v_n=\sqrt{\beta}\,u_n$ yields a constant $C_\alpha>0$ such that
\[
\sup_{n\ge n_0}\int_{\R^{2}}\bigl(e^{\alpha \beta\,u_n^{2}}-1\bigr)\,dx
=\sup_{n\ge n_0}\int_{\R^{2}}\bigl(e^{\alpha v_n^{2}}-1\bigr)\,dx
\le C_\alpha.
\]
Choose $\alpha$ sufficiently close to $4\pi$ so that
\[
q=\frac{\alpha\beta}{4\pi}>1.
\]
Then $4\pi q=\alpha\beta$ and
\[
e^{4\pi q\,u_n^{2}}-1=e^{\alpha\beta\,u_n^{2}}-1.
\]
Thus, the desired uniform bound holds for all $n\ge n_0$.
Absorbing the finitely many indices $n<n_0$ into the constant, we obtain the conclusion.
\end{proof}

\begin{Lem}\label{Lem2.6}
The maps $u\mapsto u^{+}$ and $u\mapsto u^{-}$ are continuous from $E$ to $E$.
In particular, if $u_n\to u$ in $E$, then $u_n^{\pm}\to u^{\pm}$ in $E$.
\end{Lem}

\begin{proof}
Let $\Phi_{+}(t)=t^{+}$ and $\Phi_{-}(t)=t^{-}$.
Both maps are Lipschitz on $\R$ and satisfy $\Phi_{\pm}(0)=0$.
It is a standard fact that if $\Phi:\R\to\R$ is Lipschitz and $\Phi(0)=0$, then the associated Nemytskii operator $u\mapsto \Phi(u)$ is continuous on $H^{1}(\R^{2})$ and also on $H^{s}(\R^{2})$ for every $s\in(0,1)$.
Therefore, if $u_n\to u$ in $E$, then we have
\[
u_n^\pm\to u^\pm \quad \text{in } H^{1}(\R^{2})
\qquad\text{and}\qquad
u_n^\pm\to u^\pm \quad \text{in } H^{s}(\R^{2}).
\]
Moreover, since $\Phi_{\pm}$ is $1$-Lipschitz, we get
\[
|u_n^\pm-u^\pm|\le |u_n-u|
\quad \text{a.e. in } \R^{2}.
\]
Thus, we obtain
\[
\int_{\R^{2}}V(x)\,|u_n^\pm-u^\pm|^{2}\,dx
\le
\int_{\R^{2}}V(x)\,|u_n-u|^{2}\,dx.
\]
Combining the convergence in $H^{1}(\R^{2})$ and $H^{s}(\R^{2})$ with the above weighted $L^{2}$ estimate, we conclude that
\[
u_n^\pm\to u^\pm \quad \text{in } E.
\]
This proves the claim.
\end{proof}

\subsection{Critical compactness below a threshold level}
We define the compactness threshold level $\ell_*>0$ by
\[
\ell_*:=\left(\frac12-\frac{1}{2\theta}\right)\frac{1}{p_\mu},
\]
where \(p_\mu=\frac{4}{4-\mu}.\)
\begin{Lem}\label{Lem2.7}
Assume \((V_1)\)--\((V_2)\) and \((f_1)\)--\((f_4)\). Let \((u_n)\subset E\) be a Palais--Smale sequence for \(I\) at level \(\ell\), that is,
\[
I(u_n)\to \ell,
\qquad
\|I'(u_n)\|_{E'}\to 0.
\]
Assume, in addition, that \(u_n\ge0\) a.e. in \(\R^2\) for all \(n\). If
\[
\ell<\ell_*,
\]
then \((u_n)\) is precompact in \(E\).
\end{Lem}

\begin{proof}
Set
\[
a=\frac12-\frac{1}{2\theta}>0.
\]
Since \(u_n\ge0\) a.e. and \((f_3)\) holds on \([0,+\infty)\), one has
\[
\frac{1}{2\theta}f(u_n)u_n-\frac12 F(u_n)\ge 0
\quad\text{a.e. in }\R^2.
\]
Therefore,
\[
\begin{aligned}
I(u_n)-\frac{1}{2\theta}I'(u_n)[u_n]
&=a\|u_n\|^{2}
+\int_{\R^{2}}\left(\frac{1}{|x|^{\mu}}*F(u_n)\right)
\left(\frac{1}{2\theta}f(u_n)u_n-\frac12F(u_n)\right)\,dx\\
&\ge a\|u_n\|^{2}.
\end{aligned}
\]
Since \(I(u_n)\to\ell\) and \(I'(u_n)\to0\) in \(E'\), the sequence \((u_n)\)
is bounded in \(E\), and \(I'(u_n)u_n\to0\). Moreover,
\[
\limsup_{n\to\infty}\|u_n\|^{2}\le \frac{\ell}{a}<\frac{\ell_*}{a}
=\frac1{p_\mu}.
\]
Therefore, there exists \(\delta\in(0,1)\) such that
\[
\limsup_{n\to\infty}\|\nabla u_n\|_2^2
\le
\limsup_{n\to\infty}\|u_n\|^2
\le \frac{1-\delta}{p_\mu}.
\]
Set \(v_n=\sqrt{p_\mu}\,u_n\). From Lemma~\ref{Lem2.5} there exist numbers \(q_0>1\)
and \(C_L>0\) such that
\[
\sup_n\int_{\R^2}\left(e^{4\pi q_0p_\mu u_n^2}-1\right)\,dx\le C_L.
\]
Choose \(\alpha>4\pi\), \(\sigma>1\), and \(\tau>1\), all sufficiently close to
\(4\pi,1,1\), respectively, such that
\[
\alpha\sigma\tau<4\pi q_0.
\]
Then
\begin{equation}\label{eq2.6}
\sup_n\int_{\R^2}\left(e^{\alpha\sigma\tau p_\mu u_n^2}-1\right)\,dx<+\infty .
\end{equation}
Now for every \(\varepsilon>0\) and every \(q>2\), the critical growth assumptions
give constants \(C_{\varepsilon,\alpha,q}\), \(\widetilde C_{\varepsilon,\alpha,q}\), and
\(\widehat C_{\varepsilon,\alpha,q}\) such that
\begin{equation}\label{eq2.7}
|f(t)t|
\le \varepsilon |t|^{\frac{4-\mu}{2}}
+C_{\varepsilon,\alpha,q}|t|^{q}\bigl(e^{\alpha t^2}-1\bigr),
\end{equation}
\begin{equation}\label{eq2.8}
|F(t)|
\le \varepsilon |t|^{\frac{4-\mu}{2}}
+\widetilde C_{\varepsilon,\alpha,q}|t|^{q}\bigl(e^{\alpha t^2}-1\bigr),
\end{equation}
and
\begin{equation}\label{eq2.9}
|f(t)|
\le \varepsilon |t|^{\frac{2-\mu}{2}}
+\widehat C_{\varepsilon,\alpha,q}|t|^{q-1}\bigl(e^{\alpha t^2}-1\bigr).
\end{equation}
Let \(\sigma'\) denote the H\"older conjugate of \(\sigma\). Thus, using the following inequality
\[
(e^t-1)^{p_\mu\tau}\le C(e^{p_\mu\tau t}-1),\qquad t\ge0,
\]
 combined with H\"older's inequality, \eqref{eq2.6}, and the boundedness of \((u_n)\) in all finite \(L^p\)-spaces, we obtain the higher integrability estimate
\begin{equation}\label{eq2.10}
\sup_n\|F(u_n)\|_{L^{p_\mu\tau}(\R^2)}
+
\sup_n\|f(u_n)u_n\|_{L^{p_\mu\tau}(\R^2)}
<+\infty .
\end{equation}
Indeed, we have
\[
\begin{aligned}
\int_{\R^2}|u_n|^{q p_\mu\tau}
      (e^{\alpha u_n^2}-1)^{p_\mu\tau}\,dx
&\le
C\Bigl(\int_{\R^2}(e^{\alpha\sigma\tau p_\mu u_n^2}-1)\,dx\Bigr)^{1/\sigma}
\Bigl(\int_{\R^2}|u_n|^{q p_\mu\tau\sigma'}\,dx\Bigr)^{1/\sigma'} ,
\end{aligned}
\]
and the right-hand side is uniformly bounded. Since \((u_n)\) is bounded in \(E\), after passing to a subsequence,
\[
u_n\rightharpoonup u \quad\text{in }E,\qquad
u_n\to u \quad\text{in }L^p(\R^2)\text{ for every }2\le p<+\infty,
\qquad
u_n(x)\to u(x)\quad\text{a.e. in }\R^2.
\]
The estimate \eqref{eq2.10} implies that \(\{|F(u_n)|^{p_\mu}\}_n,\) and \(\{|f(u_n)u_n|^{p_\mu}\}_n\) are uniformly integrable on every measurable set of finite measure. Hence, by Vitali's theorem, we have 
\[
F(u_n)\to F(u),\qquad f(u_n)u_n\to f(u)u
\quad\text{in }L^{p_\mu}(B_R), \,\, \text{for every}\, R>0.
\]
We now pass from local to global convergence. Let \(m_1=q p_\mu\sigma'\) and \(m_2=q p_\mu\tau\sigma'.\) Choose \(q>2\) such that both exponents are at least \(2\). Since
\(u_n\to u\) strongly in \(L^{m_1}(\R^2)\), \(L^{m_2}(\R^2)\), and \(L^2(\R^2)\),
the families \(\{|u_n|^{m_1}\}\), \(\{|u_n|^{m_2}\}\), and \(\{|u_n|^2\}\) are
tight in \(L^1(\R^2)\). Thus, for every \(\varepsilon>0\), there exists
\(R_\varepsilon>0\) such that the corresponding tails of \(u_n\) and of \(u\)
outside \(B_{R_\varepsilon}\) are bounded by \(\varepsilon\), uniformly in \(n\). Now, using \eqref{eq2.8}, \eqref{eq2.6}, and H\"older's inequality,
we get
\[
\begin{aligned}
\int_{\R^2\setminus B_{R_\varepsilon}} |F(u_n)|^{p_\mu}\,dx
&\le
C\int_{\R^2\setminus B_{R_\varepsilon}} |u_n|^2\,dx  
+C
\Bigl(\int_{\R^2}(e^{\alpha\sigma p_\mu u_n^2}-1)\,dx\Bigr)^{1/\sigma}
\Bigl(\int_{\R^2\setminus B_{R_\varepsilon}} |u_n|^{m_1}\,dx\Bigr)^{1/\sigma'}\\
&\le C\varepsilon^{1/\sigma'}+C\varepsilon .
\end{aligned}
\]
Note that the same estimate holds for \(F(u)\). Also, the identical argument based on
\eqref{eq2.7} applies to \(f(u_n)u_n\) and \(f(u)u\). Combining the
tail estimate with the local Vitali convergence yields
\[
F(u_n)\to F(u),\quad \text{and}\quad f(u_n)u_n\to f(u)u
\quad\text{in }L^{p_\mu}(\R^2).
\]
We now define
\[
B(\varphi,\psi)=\int_{\R^2}\left(\frac1{|x|^\mu}*\varphi\right)\psi\,dx .
\]
By Lemma~\ref{Lem2.3}, \(B\) is continuous on
\(L^{p_\mu}(\R^2)\times L^{p_\mu}(\R^2)\). Thus, we have
\[
\int_{\R^2}\left(\frac1{|x|^\mu}*F(u_n)\right)f(u_n)u_n\,dx
\to
\int_{\R^2}\left(\frac1{|x|^\mu}*F(u)\right)f(u)u\,dx.
\]
Since \(I'(u_n)[u_n]\to0\), we obtain
\begin{equation}\label{eq2.11}
\|u_n\|^2\to
\int_{\R^2}\left(\frac1{|x|^\mu}*F(u)\right)f(u)u\,dx .
\end{equation}
Next, we prove that \(u\) is a critical point of \(I\). Again, for \(\varphi\in C_c^\infty(\R^2)\) and \(K=\operatorname{supp}\varphi\), using the Vitali argument and \eqref{eq2.9}, we get
\[
f(u_n)\varphi\to f(u)\varphi
\quad\text{in }L^{p_\mu}(\R^2).
\]
Together with \(F(u_n)\to F(u)\) in \(L^{p_\mu}(\R^2)\), this implies
\[
\int_{\R^2}\left(\frac1{|x|^\mu}*F(u_n)\right)f(u_n)\varphi\,dx
\to
\int_{\R^2}\left(\frac1{|x|^\mu}*F(u)\right)f(u)\varphi\,dx.
\]
Passing to the limit in \(I'(u_n)[\varphi]\to0\), we get
\(I'(u)[\varphi]=0\) for all \(\varphi\in C_c^\infty(\R^2)\). Since
\(I\in C^1(E,\R)\), the functional \(I'(u)\) is continuous on \(E\).  Therefore,
by Lemma~\ref{Lem2.2}, we have \(I'(u)=0\) in \(E'.\) Thus, we get
\[
\|u\|^2=
\int_{\R^2}\left(\frac1{|x|^\mu}*F(u)\right)f(u)u\,dx.
\]
Combining this identity with \eqref{eq2.11}, we have \(\|u_n\|^2\to\|u\|^2.\) Since \(u_n\rightharpoonup u\) in the Hilbert space \(E\), it follows that
\(u_n\to u\) strongly in \(E\). This proves the compactness assertion.
\end{proof}

\section{Positive solution with least energy}\label{Sec3}
In this section we assume $(V_1)$--$(V_2)$ and $(f_1)$--$(f_4)$. We define the least energy level on the Nehari manifold as 
\[
c:=\inf_{u\in\mathcal N} I(u).
\]

\begin{Lem}\label{Lem3.1}
Let $u\in\mathcal N$ and set $g_u(t)=I(tu)$ for $t>0$. Then
\[
g_u'(t)>0\quad\text{for }0<t<1,
\qquad
g_u'(t)<0\quad\text{for }t>1.
\]
In particular,
\[
I(tu)<I(u)\quad\text{for every }t>0,\ t\ne1.
\]
\end{Lem}

\begin{proof}
Fix $u\in\mathcal N$ and define
\[
g(t)=I(tu),\qquad t>0.
\]
Therefore, we have
\[
g(t)=\frac{t^{2}}{2}\|u\|^{2}
-\frac{1}{2}\int_{\R^{2}}\left(\frac{1}{|x|^{\mu}}*F(tu)\right)F(tu)\,dx,
\]
and
\[
g'(t)=t\|u\|^{2}
-\int_{\R^{2}}\left(\frac{1}{|x|^{\mu}}*F(tu)\right) f(tu)\,u\,dx.
\]
Using the truncation condition in $(f_1)$, we have $f(\tau)=0$ and $F(\tau)=0$ for all $\tau\le0$. Thus, we get
\[
F(tu)=F(tu^{+}),
\qquad
f(tu)\,u=f(tu^{+})\,u^{+}
\quad\text{a.e. in }\R^{2},~~\text{for every}~~t>0.
\]
This implies that
\begin{equation}\label{eq3.1}
g'(t)=t\|u\|^{2}
-\int_{\R^{2}}\left(\frac{1}{|x|^{\mu}}*F(tu^{+})\right) f(tu^{+})\,u^{+}\,dx.
\end{equation}
Now, since $u\in\mathcal N$, we have $I'(u)[u]=0$, that is,
\begin{equation}\label{eq3.2}
\|u\|^{2}
=\int_{\R^{2}}\left(\frac{1}{|x|^{\mu}}*F(u)\right) f(u)\,u\,dx
=\int_{\R^{2}}\left(\frac{1}{|x|^{\mu}}*F(u^{+})\right) f(u^{+})\,u^{+}\,dx.
\end{equation}
In particular, $u^{+}\ne0$.

We now define
\[
\mathcal J(t)
=\int_{\R^{2}}\left(\frac{1}{|x|^{\mu}}*F(tu^{+})\right) f(tu^{+})\,u^{+}\,dx
=\iint_{\R^{2}\times\R^{2}}
\frac{F\bigl(tu^{+}(y)\bigr)\,f\bigl(tu^{+}(x)\bigr)\,u^{+}(x)}{|x-y|^{\mu}}\,dx\,dy.
\]
Therefore, from \eqref{eq3.1} and \eqref{eq3.2}, we deduce that
\[
g'(t)=t\|u\|^{2}-\mathcal J(t),\qquad \text{and}\qquad
\mathcal J(1)=\|u\|^{2}.
\]
Note that by $(f_2)$, the map $t\mapsto \frac{f(t)}{t}$ is strictly increasing on $(0,+\infty)$.
Therefore, for every $a>0$,
\begin{equation}\label{eq3.3}
f(ta)\le t\,f(a)\quad\text{for }0<t\le1,
\qquad
f(ta)\ge t\,f(a)\quad\text{for }t\ge1.
\end{equation}
Moreover, using
\[
F(ta)=\int_{0}^{ta}f(\tau)\,d\tau
=t\int_{0}^{a}f(t\sigma)\,d\sigma,
\]
we obtain
\begin{equation}\label{eq3.4}
F(ta)\le t^{2}F(a)\quad\text{for }0<t\le1,
\qquad
F(ta)\ge t^{2}F(a)\quad\text{for }t\ge1.
\end{equation}
We now divide the proof into two cases. Let $0<t<1$. Then for a.e. $(x,y)\in\R^{2}\times\R^{2}$,
\[
F\bigl(tu^{+}(y)\bigr)\,f\bigl(tu^{+}(x)\bigr)\,u^{+}(x)
\le t^{3}F\bigl(u^{+}(y)\bigr)\,f\bigl(u^{+}(x)\bigr)\,u^{+}(x).
\]
Since the kernel $|x-y|^{-\mu}$ is positive, it follows that
\[
\mathcal J(t)\le t^{3}\mathcal J(1)=t^{3}\|u\|^{2}.
\]
Therefore, we get
\[
g'(t)\ge t\|u\|^{2}-t^{3}\|u\|^{2}
=t(1-t^{2})\|u\|^{2}>0
\quad\text{for }0<t<1.
\]
On the other hand, if $t>1$, then the inequalities in \eqref{eq3.3} and \eqref{eq3.4} are reversed, and thus we have
\[
\mathcal J(t)\ge t^{3}\mathcal J(1)=t^{3}\|u\|^{2}.
\]
Consequently,
\[
g'(t)\le t\|u\|^{2}-t^{3}\|u\|^{2}
=t(1-t^{2})\|u\|^{2}<0,
\quad\text{for }t>1.
\]
Therefore $g$ is strictly increasing on $(0,1)$ and strictly decreasing on $(1,+\infty)$. It follows that \(g\) has a unique global maximum at \(t=1\), which implies that
\[
I(tu)=g(t)<g(1)=I(u)
\quad\text{for every }t>0,\ t\ne1.
\]
This completes the proof.
\end{proof}

\begin{Lem}\label{Lem3.2}
For each $u\in E$ with $u\ge 0$ and $u\not\equiv 0$, there exists a unique $t_u>0$ such that $t_u u \in \mathcal{N}$.
Moreover,
\[
I(t_u u)=\max_{t>0} I(tu).
\]
\end{Lem}

\begin{proof}
Fix $u\in E$ with $u\ge0$ and $u\not\equiv0$, and define \(g(t)=I(tu),\, t>0.\)
Then we have
\[
g(t)=\frac{t^{2}}{2}\|u\|^{2}
-\frac{1}{2}\int_{\R^{2}}\left(\frac{1}{|x|^{\mu}}*F(tu)\right)F(tu)\,dx,
\]
and
\[
g'(t)=t\|u\|^{2}-\int_{\R^{2}}\left(\frac{1}{|x|^{\mu}}*F(tu)\right)f(tu)\,u\,dx.
\]
Therefore, \(g'(t)=0\) if and only if \(I'(tu)[tu]=0,\) that is, \(g'(t)=0\) if and only if \(tu\in\mathcal N.\)

We first claim that $g(t)>0$ for all sufficiently small $t>0$. Indeed, using \eqref{eq2.8}, for every $\varepsilon>0$, for every $\alpha>4\pi$, and for every $q>2$, there exists $C>0$ such that
\[
|F(\tau)|
\le \varepsilon\,|\tau|^{\frac{4-\mu}{2}}
+C\,|\tau|^{q}\bigl(e^{\alpha \tau^{2}}-1\bigr)
\quad\text{for all }\tau\in\R.
\]
By Lemma~\ref{Lem2.3} with $N=2$ and $t=r=p_\mu$, we have
\[
\int_{\R^{2}}\left(\frac{1}{|x|^{\mu}}*F(tu)\right)F(tu)\,dx
\le C_\mu \|F(tu)\|_{p_\mu}^{2}.
\]
Using the above growth estimate, combined with the inequality \((a+b)^{p_\mu}\le 2^{p_\mu-1}(a^{p_\mu}+b^{p_\mu}),\)  and \(p_\mu\frac{4-\mu}{2}=2,\) we obtain
\[
\|F(tu)\|_{p_\mu}^{2}
\le C\,\varepsilon^{2}\,t^{4-\mu}
+C\,t^{2q}\Bigl\||u|^{q}\bigl(e^{\alpha t^{2}u^{2}}-1\bigr)\Bigr\|_{p_\mu}^{2}.
\]
Now choose $\sigma>1$ and let $\sigma'$ be its H\"older conjugate. Then we have
\[
\int_{\R^{2}}|u|^{q p_\mu}\bigl(e^{\alpha p_\mu t^{2}u^{2}}-1\bigr)\,dx
\le
\Bigl(\int_{\R^{2}}\bigl(e^{\alpha p_\mu \sigma t^{2}u^{2}}-1\bigr)\,dx\Bigr)^{\frac{1}{\sigma}}
\Bigl(\int_{\R^{2}}|u|^{q p_\mu\sigma'}\,dx\Bigr)^{\frac{1}{\sigma'}}.
\]
Since $u\in E\subset H^{1}(\R^{2})$, we may choose $t_0\in(0,1)$ such that
\[
\frac{\alpha p_\mu \sigma t_0^{2}}{4\pi}\,\|\nabla u\|_{2}^{2}\le 1.
\]
For $t\in(0,t_0]$, we define
\[
w_t=\left(\frac{\alpha p_\mu \sigma}{4\pi}\right)^{\frac12}tu.
\]
Thus, we have
\[
\|\nabla w_t\|_2\le1,
\qquad
\|w_t\|_2\le \left(\frac{\alpha p_\mu \sigma}{4\pi}\right)^{\frac12}t_0\|u\|_2.
\]
Therefore, by Proposition~\ref{Pro2.1}, there exists $C>0$ such that
\[
\sup_{t\in(0,t_0]}\int_{\R^{2}}\bigl(e^{\alpha p_\mu \sigma t^{2}u^{2}}-1\bigr)\,dx\le C.
\]
Since $u\in L^r(\R^2)$ for every $r\ge2$, it follows that
\[
\int_{\R^{2}}\left(\frac{1}{|x|^{\mu}}*F(tu)\right)F(tu)\,dx
\le C\,t^{4-\mu}+C\,t^{2q}~~\text{for all}~t\in(0,t_0].
\]
Therefore,
\[
g(t)\ge \frac{t^{2}}{2}\|u\|^{2}-C\,t^{4-\mu}-C\,t^{2q}
\quad\text{for all }t\in(0,t_0].
\]
Since $4-\mu>2$ and $2q>2$, there exists $t_1\in(0,t_0]$ such that
\[
g(t)>0
\quad\text{for all }t\in(0,t_1).
\]
In particular, we have
\[
\lim_{t\downarrow0}g(t)=0.
\]
Next, we prove that \(g(t)\to-\infty
\quad\text{as }t\to+\infty.\) For this purpose, we set
\[
A(t)=\int_{\R^{2}}\left(\frac{1}{|x|^{\mu}}*F(tu)\right)F(tu)\,dx.
\]
Since $u\ge0$ and $u\not\equiv0$, the set $\{x\in\R^2:u(x)>0\}$ has positive measure.
Moreover, by $(f_3)$, we get \(f(t)\ge0\) for all \(t>0\). Since \(t\mapsto f(t)/t\) is strictly increasing on \((0,+\infty)\), it follows that
\[
f(t)>0\quad\text{for all }t>0.
\]
This implies that \(F(t)>0\) for all \(t>0\) and therefore, we have \(A(t)>0\) for all \(t>0\). Using the symmetry of the convolution in $A(t)$, we compute
\[
A'(t)=2\int_{\R^{2}}\left(\frac{1}{|x|^{\mu}}*F(tu)\right)f(tu)\,u\,dx.
\]
By $(f_3)$ and $u\ge0$, we get
\[
f(tu)\,u=\frac{f(tu)\,tu}{t}\ge \frac{\theta}{t}F(tu).
\]
Therefore, we have
\[
A'(t)\ge \frac{2\theta}{t}A(t)
\quad\text{for all }t>0.
\]
Integrating this differential inequality on $[1,t]$, we obtain
\[
A(t)\ge A(1)\,t^{2\theta}~~\text{for}~~t\geq1.
\]
Consequently, using the fact $\theta>1$, we get
\[
g(t)=\frac{t^{2}}{2}\|u\|^{2}-\frac{1}{2}A(t)
\le \frac{t^{2}}{2}\|u\|^{2}-\frac{A(1)}{2}\,t^{2\theta}
\to -\infty
\quad\text{as }t\to+\infty.
\]
Note that $g$ is continuous on $(0,+\infty)$. We extend it continuously to $[0,+\infty)$ by setting $g(0)=0$. Since $g(t)>0$ for $t\in(0,t_1)$ and $g(t)\to -\infty$ as $t\to+\infty$, the function $g$ attains its global maximum at some $t_u>0$. Hence $g'(t_u)=0$, and therefore
\[
t_u u\in\mathcal N,
\qquad
I(t_u u)=g(t_u)=\max_{t>0}I(tu).
\]
It remains to prove the uniqueness of $t_u$. Assume that $t_1u,t_2u\in\mathcal N$ with $t_1\ne t_2$.
Without loss of generality, let $0<t_1<t_2$, and set
\[
v=t_1u\in\mathcal N.
\]
Then we have
\[
t_2u=\frac{t_2}{t_1}v,
\qquad
\frac{t_2}{t_1}>1.
\]
By Lemma~\ref{Lem3.1}, we get
\[
I(t_2u)=I\left(\frac{t_2}{t_1}v\right)<I(v)=I(t_1u).
\]
On the other hand, applying Lemma~\ref{Lem3.1} to $t_2u\in\mathcal N$, we obtain
\[
I(t_1u)=I\left(\frac{t_1}{t_2}(t_2u)\right)<I(t_2u),
\]
which is impossible.
Therefore, $t_u$ is unique.
\end{proof}

\begin{Lem}\label{Lem3.3}
There exists a constant $C>0$ such that $\|u\|^{2}\ge C$ for every $u\in\mathcal{N}.$
\end{Lem}

\begin{proof}
Assume by contradiction that there exists $(u_n)\subset\mathcal N$ such that $\|u_n\|\to 0.$ By $(V_1)$, we have
\[
\|u_n\|_{H^{1}(\R^{2})}^{2}
\le C\|u_n\|^{2}\to0.
\]
Therefore, 
\[
u_n\to0 \quad\text{in }L^{p}(\R^{2})\ \text{for every }p\in[2,+\infty),
\qquad
\|u_n\|_{p}\le C_p\|u_n\|.
\]
Since $u_n\in\mathcal N$, we have
\[
\|u_n\|^{2}
=\int_{\R^{2}}\left(\frac{1}{|x|^{\mu}}*F(u_n)\right) f(u_n)\,u_n\,dx.
\]
Thus, by Lemma~\ref{Lem2.3}, we obtain
\[
\|u_n\|^{2}
\le C_\mu\,\|F(u_n)\|_{r}\,\|f(u_n)u_n\|_{r},
\]
where $r=p_\mu=\frac{4}{4-\mu}.$ Now, fix $\alpha>4\pi$, $q>2$, and $\kappa>1$.
By \eqref{eq2.7} and \eqref{eq2.8}, for every $\varepsilon>0$ there exists $C>0$ such that
\[
|F(t)|+|f(t)t|
\le \varepsilon |t|^{\frac{4-\mu}{2}}
+C\,|t|^{q}\bigl(e^{\alpha t^{2}}-1\bigr)
\quad\text{for all }t\in\R.
\]
Since $r\frac{4-\mu}{2}=2,$ we get
\[
\bigl\||u_n|^{\frac{4-\mu}{2}}\bigr\|_{r}
=\|u_n\|_{2}^{\frac{4-\mu}{2}}
\le C\|u_n\|^{\frac{4-\mu}{2}}.
\]
Let $\kappa'$ be the H\"older conjugate to $\kappa$.
Since $\|u_n\|_{H^{1}(\R^{2})}\to0$, for $n$ large, we have $\alpha r \kappa\,\|u_n\|_{H^{1}(\R^{2})}^{2}\le 2\pi.$ Now, set
\[
w_n=\frac{u_n}{\|u_n\|_{H^{1}(\R^{2})}}.
\]
Therefore,
\[
\|\nabla w_n\|_2\le1,
\qquad
\|w_n\|_2\le1.
\]
By Proposition~\ref{Pro2.1}, we have
\[
\int_{\R^{2}}\bigl(e^{2\pi w_n^{2}}-1\bigr)\,dx\le C.
\]
Therefore, we obtain
\[
\int_{\R^{2}}\bigl(e^{\alpha r \kappa\,u_n^{2}}-1\bigr)\,dx\le C.
\]
Moreover, for every $\rho>1$ there exists $C_\rho>0$ such that
\[
(e^t-1)^\rho\le C_\rho(e^{\rho t}-1),\qquad t\ge0.
\]
Using this estimate with $\rho=r$ and then H\"older's inequality, we get
\[
\begin{aligned}
\bigl\||u_n|^{q}(e^{\alpha u_n^{2}}-1)\bigr\|_{r}^{r}
&=\int_{\R^2}|u_n|^{qr}\bigl(e^{\alpha u_n^2}-1\bigr)^r\,dx \\
&\le C\int_{\R^2}|u_n|^{qr}\bigl(e^{\alpha r u_n^2}-1\bigr)\,dx \\
&\le C
\Bigl(\int_{\R^{2}}\bigl(e^{\alpha r \kappa\,u_n^{2}}-1\bigr)\,dx\Bigr)^{\frac{1}{\kappa}}
\Bigl(\int_{\R^{2}}|u_n|^{qr\kappa'}\,dx\Bigr)^{\frac{1}{\kappa'}}.
\end{aligned}
\]
Thus, we have
\[
\bigl\||u_n|^{q}(e^{\alpha u_n^{2}}-1)\bigr\|_{r}
\le C\|u_n\|_{qr\kappa'}^{q}
\le C\|u_n\|^{q}.
\]
It follows that
\[
\|F(u_n)\|_{r}+\|f(u_n)u_n\|_{r}
\le C\Bigl(\|u_n\|^{\frac{4-\mu}{2}}+\|u_n\|^{q}\Bigr).
\]
Therefore, 
\[
\|u_n\|^{2}
\le C\Bigl(\|u_n\|^{\frac{4-\mu}{2}}+\|u_n\|^{q}\Bigr)^{2}
\le C\Bigl(\|u_n\|^{4-\mu}+\|u_n\|^{2q}\Bigr).
\]
Dividing by $\|u_n\|^{2}$, we get
\[
1\le C\Bigl(\|u_n\|^{2-\mu}+\|u_n\|^{2q-2}\Bigr)\to0,
\]
which is impossible. Thus $\mathcal N$ is bounded away from zero.
\end{proof}

\begin{Lem}\label{Lem3.4}
If $(u_n)\subset\mathcal{N}$ is a minimizing sequence for the level $c$, then $(u_n)$ is bounded in $E$.
\end{Lem}

\begin{proof}
Let $(u_n)\subset\mathcal N$ be such that $I(u_n)\to c.$ By $(f_3)$ for \(t\ge0\) and by the truncation condition in $(f_1)$ for \(t\le0\), we have
\[
0\le \theta F(t)\le f(t)t
\quad\text{for all }t\in\R.
\]
Therefore
\[
\frac{1}{2\theta}f(t)t-\frac12F(t)\ge 0
\quad\text{for all }t\in\R.
\]
Since $u_n\in\mathcal N$, we have $I'(u_n)[u_n]=0$. Thus
\[
\begin{aligned}
I(u_n)
&=I(u_n)-\frac{1}{2\theta}I'(u_n)[u_n]\\
&=\left(\frac12-\frac{1}{2\theta}\right)\|u_n\|^{2}
+\int_{\R^{2}}\left(\frac{1}{|x|^{\mu}}*F(u_n)\right)
\left(\frac{1}{2\theta}f(u_n)u_n-\frac12F(u_n)\right)\,dx\\
&\ge \left(\frac12-\frac{1}{2\theta}\right)\|u_n\|^{2}.
\end{aligned}
\]
Since $I(u_n)\to c\in\R$, it follows that $(\|u_n\|)$ is bounded. Therefore, $(u_n)$ is bounded in $E$.
\end{proof}

\begin{Lem}\label{Lem3.5}
There exists a positive constant \(C_*\) such that, if \((f_4)\) holds with
\(C_0\ge C_*\), then $c<\ell_*,$ where the constant \(C_*\) is to be determined by the data and by the Moser construction in the proof below.
\end{Lem}

\begin{proof}
We construct a nonnegative test function whose one-dimensional maximum in the  Nehari manifold
is below \(\ell_*\). Choose \(x_0\in\R^2\) and \(r>0\). Since \(V\) is
continuous, it is bounded on \(\overline{B_r(x_0)}\). Suppose
\(\{\omega_k\}\subset C_c^\infty(B_r(x_0))\) is a standard nonnegative Moser sequence such that
\[
\int_{\R^2}|\nabla \omega_k|^2\,dx=1+o(1),\qquad
\|\omega_k\|_2^2=o(1),\qquad [\omega_k]_s^2=o(1),~~\text{and}~~\inf_{B_{r/k}(x_0)}\omega_k\to+\infty.\]
Then the fractional estimate follows from
\([v]_s\le C\|v\|_2^{1-s}\|\nabla v\|_2^s\). Set $\psi_k=\frac{\omega_k}{\|\omega_k\|}.$ Therefore, we have \(\psi_k\ge0\), \(\psi_k\not\equiv0\), \(\|\psi_k\|=1\), and
\[
M_k:=\inf_{B_{r/k}(x_0)}\psi_k\to+\infty .
\]
We now choose a sufficiently small \(\rho>0\) such that $\frac{\rho^2}{2}<\ell_*.$ Then choose a large \(k\) so that $\rho M_k\ge T_0.$ Fix this \(k\) and define
\[
D_k=
\iint_{B_{r/k}(x_0)\times B_{r/k}(x_0)}
\frac{dx\,dy}{|x-y|^\mu}>0,\qquad
B_k=8\pi M_k^2.
\]
We now set
\[
C_*=
\left(D_k B_k e^{B_k\rho^2}\right)^{-1/2}.
\]
Thus, for \(0\le t\le\rho\), the Choquard term is nonnegative, and putting \(e_0=\psi_k\), we get
\begin{equation}\label{eq3.5}
    I(te_0)\le\frac{t^2}{2}\le\frac{\rho^2}{2}<\ell_*.
\end{equation}
Let \(t\ge\rho\). Since \(te_0\ge T_0\) on \(B_{r/k}(x_0)\), assumption \((f_4)\)
with \(C_0\ge C_*\) gives
\[
F(te_0(x))\ge C_0 e^{4\pi t^2e_0(x)^2}
\qquad\text{for }x\in B_{r/k}(x_0).
\]
Therefore,
\[
\int_{\R^2}\left(\frac1{|x|^\mu}*F(te_0)\right)F(te_0)\,dx
\ge
C_0^2D_k e^{8\pi M_k^2t^2}.
\]
Consequently, we have
\[
I(te_0)
\le
H(t):=\frac{t^2}{2}-\frac{C_0^2D_k}{2}e^{B_kt^2}
\qquad (t\ge\rho).
\]
Since $H'(t)=t\left(1-C_0^2D_kB_ke^{B_kt^2}\right)$, the definition of \(C_*\) implies \(H'(t)\le0\) for all \(t\ge\rho\). Thus,
\begin{equation}\label{eq3.6}
    \sup_{t\ge\rho}I(te_0)\le H(\rho)
=
\frac{\rho^2}{2}-\frac{C_0^2D_k}{2}e^{B_k\rho^2}
<\ell_*.
\end{equation}
Therefore, from \eqref{eq3.5} and \eqref{eq3.6}, we have
\[
\max_{t\ge0}I(te_0)<\ell_*.
\]
Let \(t_{e_0}>0\) be given by Lemma~\ref{Lem3.2}. Then, we have $c\le I(t_{e_0}e_0)=\max_{t\ge0}I(te_0)<\ell_*.$
\end{proof}

\begin{Lem}\label{Lem3.6}
Assume $(V_1)$--$(V_2)$ and $(f_1)$--$(f_4)$.
There exists \(u\in\mathcal{N}\), with \(u\ge0\) a.e. in \(\R^2\), such that \(I(u)=c\).
\end{Lem}

\begin{proof}
Let $(u_n)\subset\mathcal N$ be a minimizing sequence, that is, $I(u_n)\to c.$ Set $w_n=u_n^{+}$. By the truncation condition in $(f_1)$,
\[
F(u_n)=F(w_n),\qquad f(u_n)u_n=f(w_n)w_n \quad \text{a.e. in }\R^2.
\]
Moreover, Lemma~\ref{Lem2.6} gives $\|w_n\|\le \|u_n\|$. Thus, for every \(t>0\),
the Choquard parts of \(I(tw_n)\) and \(I(tu_n)\) coincide, while the quadratic part
of \(I(tw_n)\) is less than that of \(I(tu_n)\). Therefore, we have
\[
I(tw_n)\le I(tu_n)\qquad\text{for all }t>0.
\]
Let $t_n>0$ be such that $t_n w_n\in\mathcal N$; the existence and uniqueness follow from Lemma~\ref{Lem3.2}. Then
\[
I(t_n w_n)=\max_{t>0}I(tw_n)\le \max_{t>0}I(tu_n)=I(u_n),
\]
where the last equality follows from Lemma~\ref{Lem3.1}. Replacing $u_n$ by $t_n w_n$, we may assume that
\[
u_n\ge0,\qquad u_n\in\mathcal N,\qquad\text{and}\qquad I(u_n)\to c.
\]
By Lemma~\ref{Lem3.4}, we get that $(u_n)$ is bounded in $E$. Since $u_n\in\mathcal N$,
\[
I(u_n)=I(u_n)-\frac{1}{2\theta}I'(u_n)[u_n]\ge a\|u_n\|^2,
\]
where $a=\frac12-\frac{1}{2\theta}.$ Therefore, using Lemma~\ref{Lem3.5}, $c<\ell_*=\frac{a}{p_\mu}$, we get
\[
\limsup_{n\to\infty}\|u_n\|^2\le \frac{c}{a}<\frac1{p_\mu}.
\]
Therefore, passing to a subsequence, there exist $\delta\in(0,1)$ and $n_0\in\N$ such that
\[
p_\mu\|\nabla u_n\|_2^2\le p_\mu\|u_n\|^2\le 1-\delta
\quad\text{for all }n\ge n_0.
\]
Moreover, up to a subsequence, we have
\[
u_n\rightharpoonup u_0 \quad\text{in }E,
\qquad
u_n\to u_0 \quad\text{in }L^p(\R^2)\ \text{for every }2\le p<\infty,
\qquad
u_n(x)\to u_0(x)\quad\text{a.e. in }\R^2,
\]
for some $u_0\in E$, with $u_0\ge0$ a.e. Arguing exactly as in the proof of Lemma~\ref{Lem2.7}, the uniform subcritical gradient bound implies
\[
F(u_n)\to F(u_0)\ \text{in }L^{p_\mu}(\R^2),
\qquad
f(u_n)u_n\to f(u_0)u_0\ \text{in }L^{p_\mu}(\R^2).
\]
By Lemma~\ref{Lem2.3},
\[
\int_{\R^2}\left(\frac1{|x|^\mu}*F(u_n)\right)F(u_n)\,dx
\to
\int_{\R^2}\left(\frac1{|x|^\mu}*F(u_0)\right)F(u_0)\,dx,
\]
and
\[
\int_{\R^2}\left(\frac1{|x|^\mu}*F(u_n)\right)f(u_n)u_n\,dx
\to
\int_{\R^2}\left(\frac1{|x|^\mu}*F(u_0)\right)f(u_0)u_0\,dx.
\]
Since $u_n\in\mathcal N$, we have
\[
\|u_n\|^2=\int_{\R^2}\left(\frac1{|x|^\mu}*F(u_n)\right)f(u_n)u_n\,dx,
\]
and therefore,
\[
\lim_{n\to\infty}\|u_n\|^2
=
\int_{\R^2}\left(\frac1{|x|^\mu}*F(u_0)\right)f(u_0)u_0\,dx.
\]
Observe that if $u_0\equiv0$, then $\|u_n\|\to0$, contradicting Lemma~\ref{Lem3.3}. Thus $u_0\not\equiv0$.

Let $t_0>0$ be the unique number such that $t_0u_0\in\mathcal N$, given by Lemma~\ref{Lem3.2}. Define $g_0(t)=I(tu_0)$. Now, by weak lower semicontinuity of the norm, we get
\[
\|u_0\|^2\le \liminf_{n\to\infty}\|u_n\|^2
=
\int_{\R^2}\left(\frac1{|x|^\mu}*F(u_0)\right)f(u_0)u_0\,dx.
\]
Thus, we have
\[
g_0'(1)=I'(u_0)[u_0]\le0.
\]
We next justify that \(t_0\le1\). Applying the proof of Lemma~\ref{Lem3.1}
to the element \(t_0u_0\in\mathcal N\), the map
\(\lambda\mapsto I(\lambda t_0u_0)\) satisfies
\[
\frac{d}{d\lambda}I(\lambda t_0u_0)>0\quad\text{for }0<\lambda<1,
\qquad \text{and} \qquad
\frac{d}{d\lambda}I(\lambda t_0u_0)<0\quad\text{for }\lambda>1.
\]
Since
\[
g_0(t)=I(tu_0)=I\left(\frac{t}{t_0}t_0u_0\right),
\]
it follows that \(g_0'(t)>0\) for \(0<t<t_0\) and \(g_0'(t)<0\) for
\(t>t_0\). If \(t_0>1\), then \(g_0'(1)>0\), contradicting
\(g_0'(1)\le0\). Therefore, \(t_0\le1\).

Because $u_n\in\mathcal N$, Lemma~\ref{Lem3.1} yields $I(t_0u_n)\le I(u_n)\quad\text{for all }n.$ Since $0<t_0\le1$, the sequence $(t_0u_n)$ still satisfies the same uniform subcritical gradient bound. Therefore, arguing as above,
\[
F(t_0u_n)\to F(t_0u_0)\qquad\text{in }L^{p_\mu}(\R^2).
\]
From Lemma~\ref{Lem2.3}, we get
\[
\int_{\R^2}\left(\frac1{|x|^\mu}*F(t_0u_n)\right)F(t_0u_n)\,dx
\to
\int_{\R^2}\left(\frac1{|x|^\mu}*F(t_0u_0)\right)F(t_0u_0)\,dx.
\]
Also, using the weak lower semicontinuity of the norm, we obtain
\[
I(t_0u_0)\le \liminf_{n\to\infty}I(t_0u_n)
\le \limsup_{n\to\infty}I(t_0u_n)
\le \lim_{n\to\infty}I(u_n)=c.
\]
Since $t_0u_0\in\mathcal N$, we also have $c\le I(t_0u_0)$. Therefore, we get $I(t_0u_0)=c.$ Setting $u=t_0u_0$ completes the proof.
\end{proof}

\begin{proof}[Proof of Theorem~\ref{Thm1.1}]
Assume the hypotheses of Theorem~\ref{Thm1.1}. By Lemma~\ref{Lem3.6}, there exists
\(u\in\mathcal N\), with \(u\ge0\) a.e. in \(\R^2\), such that $I(u)=c.$

\noindent\textbf{Step 1. \(u\) is a critical point of \(I\).}
Assume by contradiction that \(I'(u)\ne0\) in \(E'\). Then there exists
\(\phi\in E\) such that $I'(u)[\phi]<0$. Replacing \(\phi\) by a positive multiple, we may assume $I'(u)[\phi]\le -2.$ By continuity of \(I'\), there exists \(\varepsilon\in(0,1)\) such that
\[
I'(tu+\sigma\phi)[\phi]\le -1
\quad\text{for all } |t-1|\le \varepsilon,\ |\sigma|\le \varepsilon .
\]
Taking \(\varepsilon>0\) smaller if necessary, we may also assume that
\[
\varepsilon\|\phi\|<(1-\varepsilon)\|u\|.
\]
Choose \(\eta\in C^\infty([0,+\infty))\) such that
\[
0\le \eta\le1,\qquad
\eta(t)=1\text{ for }|t-1|\le\frac{\varepsilon}{2},
\qquad
\eta(t)=0\text{ for }|t-1|\ge\varepsilon .
\]
Now, define
\[
h(t)=tu+\varepsilon\eta(t)\phi,\qquad\text{and}\qquad
\Psi(t)=I(h(t)),~ t>0.
\]
Then \(h(t)\ne0\) for \(t\in[1-\varepsilon,1+\varepsilon]\). Next, we prove that
\begin{equation}\label{eq3.7}
\sup_{t>0}\Psi(t)<I(u)=c.
\end{equation}
Since \(u\in\mathcal N\), Lemma~\ref{Lem3.2} and Lemma~\ref{Lem3.1} imply
\[
I(tu)\to0\quad\text{as }t\to0^+,
\qquad\text{and}\qquad
I(tu)\to-\infty\quad\text{as }t\to+\infty.
\]
Thus there exist \(0<a<1-\varepsilon\) and \(b>1+\varepsilon\) such that
\begin{equation}\label{eq3.8}
    I(tu)\le\frac{c}{2}
\quad\text{for }t\in(0,a]\cup[b,+\infty).
\end{equation}
Note that on this set \(h(t)=tu\), and therefore, \(\Psi(t)\le c/2\). Now on the compact interval \([a,b]\) we have \(\Psi(t)<I(u)\) for every \(t\).
Indeed, if \(|t-1|\ge\varepsilon\), then \(h(t)=tu\) and Lemma~\ref{Lem3.1} give \(I(tu)<I(u)\). If \(|t-1|<\varepsilon\), then
\[
\begin{aligned}
\Psi(t)
&=I(tu)+\varepsilon\eta(t)\int_0^1
I'(tu+\sigma\varepsilon\eta(t)\phi)[\phi]\,d\sigma\\
&\le I(tu)-\varepsilon\eta(t)<I(u),
\end{aligned}
\]
where the strict inequality follows from Lemma~\ref{Lem3.1} if \(t\ne1\), combined with the fact \(\eta(1)=1\) if \(t=1\). Since \(\Psi\) is continuous and \([a,b]\) is compact, we get
\begin{equation}\label{eq3.9}
    \sup_{t\in[a,b]}\Psi(t)<I(u).
\end{equation}
Thus, the estimate \eqref{eq3.7} follows from the estimates \eqref{eq3.8} and \eqref{eq3.9}.

Again, define $\Upsilon(t)=I'(h(t))[h(t)],~ t\in[1-\varepsilon,1+\varepsilon].$ Since \(h(1-\varepsilon)=(1-\varepsilon)u\) and
\(h(1+\varepsilon)=(1+\varepsilon)u\), the proof of Lemma~\ref{Lem3.1} gives
\[
\Upsilon(1-\varepsilon)>0,\qquad\text{and}\qquad
\Upsilon(1+\varepsilon)<0.
\]
Then there exists \(\bar t\in(1-\varepsilon,1+\varepsilon)\) such that $\Upsilon(\bar t)=0.$ Since \(h(\bar t)\ne0\), we have \(h(\bar t)\in\mathcal N\). Thus, we have
\[
c\le I(h(\bar t))\le \sup_{t>0}\Psi(t)<c,
\]
which is a contradiction. Therefore, we conclude \(I'(u)=0\).

\noindent\textbf{Step 2. \(u\ge0\) a.e. in \(\R^2\).} Let \(u^+=\max\{u,0\}\ge0\) and \(u^-=\min\{u,0\}\le0\) be the positive and negative parts of $u$. We recall the following algebraic inequality: $(a-b)(a^--b^-)\ge |a^--b^-|^2$. Now, testing \(I'(u)=0\) by \(u^-\) and using the truncation condition in \((f_1)\), we obtain
\[
0=I'(u)[u^-]=(u,u^-).
\]
Moreover, we have
\[
\nabla u\cdot\nabla u^-=|\nabla u^-|^2,\qquad\text{and}\qquad uu^-=|u^-|^2.
\]
Thus, we get
\[
0=(u,u^-)\ge \|u^-\|^2,
\]
implying that \(u^-=0\). Therefore, \(u\ge0\) a.e. in \(\R^2\).

Since $u\in\mathcal N$, $u\not\equiv0$, $I'(u)=0$, and
$I(u)=c=\inf_{\mathcal N}I$, the function $u$ is a nontrivial nonnegative weak
solution of \eqref{eq1.1}. Moreover, every nontrivial weak solution $v$ of
\eqref{eq1.1} satisfies $I'(v)[v]=0$, and thus belongs to \(\mathcal N\).
Therefore, $u$ has the least energy among all nontrivial weak solutions. We now
prove that $u$ is strictly positive.

\noindent\textbf{Step 3. Strict positivity.}
Using $(f_3)$, we have $f(t)t\ge0$ for all $t\ge0$. This implies $f(t)\ge0$ for $t>0$.
We claim that
\[
f(t)>0,~\text{for all }t>0.
\]
Indeed, if $f(t_0)=0$ for some $t_0>0$, then $f(t_0)/t_0=0$. Since the map
$t\mapsto f(t)/t$ is strictly increasing on $(0,+\infty)$, we get 
$f(t)/t<0$ for every $t\in(0,t_0)$, contradicting the assumption $f(t)\ge0$ on
$(0,+\infty)$. Therefore, $f(t)>0,~\text{for all }t>0.$ Consequently, we have $F(t)>0$ for all $t>0,$ since $F$ is the primitive of $f$. $u\ge0$ and $u\not\equiv0$ imply that $F(u)\ge0$ and $F(u)\not\equiv0.$ Define
\[
H(x)=\left(\frac1{|x|^\mu}*F(u)\right)(x)f(u(x))-V(x)u(x).
\]
Since $u\in\mathcal N$ and $I(u)=c<\ell_*$, we have
\begin{align*}
    &I(u)=I(u)-\frac1{2\theta}I'(u)[u]
\ge \left(\frac12-\frac1{2\theta}\right)\|u\|^2\\
\Rightarrow&\|u\|^2<\frac1{p_\mu}.
\end{align*}
Choose $m>p_\mu$ close to $p_\mu$ and $\alpha>4\pi$ close to $4\pi$ such that $\alpha m\|\nabla u\|_2^2<4\pi.$ Then the Trudinger--Moser estimates used in Lemma~\ref{Lem2.7} imply that $f(u)\in L^m_{\mathrm{loc}}(\R^2).$ Moreover, proceeding to obtain \eqref{eq2.4} as in Lemma~\ref{Lem2.4}, we get $F(u)\in L^{p_\mu}(\R^2).$ By the mapping form of the Hardy--Littlewood--Sobolev inequality,
\[
\frac1{|x|^\mu}*F(u)\in L^{4/\mu}(\R^2).
\]
Since \(m>p_\mu=4/(4-\mu)\), H\"older's inequality gives
\[
\left(\frac1{|x|^\mu}*F(u)\right)f(u)\in L^r_{\mathrm{loc}}(\R^2)~~\text{for some \(r>1\).}
\]
Furthermore, we have \(V\in C(\R^2)\subset L^\infty_{\mathrm{loc}}(\R^2)\), and
\(u\in H^1_{\mathrm{loc}}(\R^2)\subset L^q_{\mathrm{loc}}(\R^2),~\forall\,q<+\infty\). Thus, we get
\[
V(\cdot)u\in L^r_{\mathrm{loc}}(\R^2)~~\text{for some \(r>1\).}
\]
Therefore, we obtain $H\in L^r_{\mathrm{loc}}(\R^2)$ for some \(r>1\). Now, fix \(x_0\in\R^2\) and choose \(R\in(0,1)\). Since \(u\in H^1(\R^2)\), we get \(u\in W^{1,2}_{\mathrm{loc}}(B_{4R}(x_0))\). In addition, we also have
\[
\int_{\R^2}\frac{|u(y)|}{(1+|y|)^{2+2s}}\,dy
\le
\|u\|_{2}\left(\int_{\R^2}\frac{dy}{(1+|y|)^{4+4s}}\right)^{1/2}
<\infty.
\]
Thus \(u\) belongs to the natural tail space (see \cite{GarainLindgren2023, ShangZhang2023}). We now apply Theorem~1.4 of
\cite{GarainLindgren2023} to the following problem
\[
-\Delta u+(-\Delta)^s u = H
\qquad\text{in }B_{4R}(x_0),
\]
with \(p=2\), \(A=1\), \(N=2\), and \(r>1=N/2\). Therefore, we obtain \(u\in C^\delta_{\mathrm{loc}}(\R^2)\) for some \(\delta\in(0,1)\). We now show that the following Choquard potential
\[
\Phi(x)=\left(\frac1{|x|^\mu}*F(u)\right)(x)
\]
is continuous in \(\R^2\). Let \(K\subset\R^2\) be compact. Choose \(R_K>0\) such that \(K\subset B_{R_K}\). We split the following integral into two parts:
\[
\Phi(x)
=
\int_{B_{2R_K}}\frac{F(u(y))}{|x-y|^\mu}\,dy
+
\int_{\R^2\setminus B_{2R_K}}\frac{F(u(y))}{|x-y|^\mu}\,dy
=:\Phi_1(x)+\Phi_2(x).
\]
Now, since \(u\in C(\overline{B_{2R_K}})\), the function \(F(u)\chi_{B_{2R_K}}\) is
bounded and has compact support. Again, since \(|x|^{-\mu}\in L^1(B_{3R_K})\), the continuity of convolutions with compactly supported \(L^\infty\)-data gives $\Phi_1\in C(K).$ On the other hand, if \(x\in K\) and \(y\notin B_{2R_K}\), then we have \(|x-y|\ge |y|/2\). Therefore, there exists $C>0$ such that
\[
\frac{F(u(y))}{|x-y|^\mu}
\le
C\frac{F(u(y))}{|y|^\mu}.
\]
Using \(F(u)\in L^{p_\mu}(\R^2)\), \(|y|^{-\mu}\in L^{4/\mu}(\R^2\setminus B_{2R_K})\), and H\"older's inequality, we obtain
\[
\frac{F(u(y))}{|y|^\mu}\in L^1(\R^2\setminus B_{2R_K}).
\]
By applying the dominated convergence theorem, we get \(\Phi_2\in C(K)\). Thus, we have $\Phi\in C(K)$. Since \(K\) is arbitrary, we conclude $\Phi\in C(\R^2)$. Next, we define
\[
\psi(t)=
\begin{cases}
\dfrac{f(t)}{t}, & t>0,\\[0.4em]
0, & t=0.
\end{cases}
\]
Because \(f\in C^1(\R)\), \(f(0)=f'(0)=0\), and \(f(t)=0\) for \(t\le0\), the map \(\psi\) is continuous on \([0,+\infty)\). Since \(u\ge0\) and \(u\in C(\R^2)\), the coefficient \(c(x)=\Phi(x)\psi(u(x))-V(x)\in C(\R^2)\). Moreover, \(u\) satisfies
\[
-\Delta u+(-\Delta)^su = c(x)u
\qquad\text{in }\R^2
\]
weakly. Let $c_-(x)=\min\{c(x),0\}\le0.$ Then \(c_-\in C(\R^2)\), and for every nonnegative \(\varphi\in C_c^\infty(\R^2)\),
\[
\int_{\R^2}\nabla u\cdot \nabla\varphi\,dx
+\langle (-\Delta)^su,\varphi\rangle
=
\int_{\R^2}c(x)u\varphi\,dx
\ge
\int_{\R^2}c_-(x)u\varphi\,dx.
\]
Therefore, \(u\) is a continuous weak supersolution in every ball \(B_\rho(x_0)\) to the problem
\[
-\Delta v+(-\Delta)^sv = c_-(x)v
\]
with \(u\ge0\) a.e. in \(\R^2\setminus B_\rho(x_0)\). We now recall the definition of $X^{1,2}(\Omega)$ from \cite{ShangZhang2023}: \textit{For a bounded subset $\Omega$ of $\R^2$, we say that
$u \in X^{1,2}(\Omega)$ if and only if $u \in L^2_{\mathrm{loc}}(\mathbb{R}^2)$ and there exists an open set $U \supset \Omega$ such that} 
$$\|u\|_{H^{1}(U)} +
\int_{\mathbb{R}^2}\frac{|u(x)|}{(1+|x|)^{2+2s}}\,dx< \infty.$$
Since \(u\in H^1(B_\rho(x_0))\) and the above weighted \(L^1\)-estimate holds, we conclude that \(u\in X^{1,2}(B_\rho(x_0))\). Therefore, by Theorem~1.4(i) of \cite{ShangZhang2023}, we get that for every \(x_0\in\R^2\) and for every \(\rho>0\) either $u>0$ in $B_\rho(x_0)$ or \(u=0\) a.e. in \(\R^2\). Since \(u\in\mathcal N\), we have \(u\not\equiv0\). Therefore, \(u>0\) in every ball in $\R^2$. Since \(x_0\in\R^2\) is arbitrary, we conclude that $u>0$ in $\R^2$. Thus $u$ is a least energy positive solution of \eqref{eq1.1}.
\end{proof}

\section*{Acknowledgments}

\medskip
{\bf Funding:} This work is supported by National Natural Science Foundation of China (12301145, 12561020, 12261107) and Yunnan Fundamental Research Projects (202401AU070123, 202601AT070048). S. Ghosh acknowledges financial support under the ARG-MATRICS grant from ANRF, India, (Grant number: ANRF/ARGM/2025/001570/MTR).

\medskip
{\bf Author Contributions:} All authors contributed equally to the writing and preparation of the manuscript.

\medskip
{\bf Data availability:}  Data sharing is not applicable to this article as no new data were created or analyzed in this study.

\medskip
{\bf Conflict of Interests:} The authors declare that they have no conflict of interest.

\end{document}